\documentclass[10pt]{article}
\font\smallit=cmti10

\usepackage{color}
\usepackage{hyperref}
\usepackage{color}
\usepackage{breakurl}
\usepackage{comment}
\newcommand{\bburl}[1]{\textcolor{blue}{\url{#1}}}

\usepackage{amssymb,latexsym,amsmath,epsfig,amsthm} %% Add other packages as necessary

\makeatletter

\renewcommand\section{\@startsection {section}{1}{\z@}
{-30pt \@plus -1ex \@minus -.2ex}
{2.3ex \@plus.2ex}
{\normalfont\normalsize\bfseries\boldmath}}

\renewcommand\subsection{\@startsection{subsection}{2}{\z@}
{-3.25ex\@plus -1ex \@minus -.2ex}
{1.5ex \@plus .2ex}
{\normalfont\normalsize\bfseries\boldmath}}

\renewcommand{\@seccntformat}[1]{\csname the#1\endcsname. }

\makeatother

\newtheorem{thm}{Theorem}[section]
\newtheorem{conj}[thm]{Conjecture}
\newtheorem{cor}[thm]{Corollary}
\newtheorem{lem}[thm]{Lemma}
\newtheorem{prop}[thm]{Proposition}
\newtheorem{exa}[thm]{Example}

\begin{document}

\begin{center}
\uppercase{\bf Union of Two Arithmetic Progressions with the Same Common Difference Is Not Sum-dominant}
\vskip 20pt
{\bf H\`ung Vi\d{\^e}t Chu}\\

\end{center}
\vskip 20pt
\centerline{\smallit Received: , Revised: , Accepted: , Published: } % We will fill in the dates
\vskip 30pt

\centerline{\bf Abstract}
\noindent Given a finite set $A\subseteq \mathbb{N}$, define the sum set
$$A+A = \{a_i+a_j\mid a_i,a_j\in A\}$$ and the difference set 
$$A-A = \{a_i-a_j\mid a_i,a_j\in A\}.$$ The set $A$ is said to be sum-dominant
if $|A+A|>|A-A|$. We prove the following results
\begin{enumerate}
    \item The union of two arithmetic progressions (with the same common difference) is not sum-dominant. This result partially proves a conjecture proposed by the author in a previous paper; that is, the union of any two arbitrary arithmetic progressions is not sum-dominant. 
    \item Hegarty proved that a sum-dominant set must have at least $8$ elements with computers' help. The author of the current paper provided a human-verifiable proof that a sum-dominant set must have at least $7$ elements. A natural question is about the largest cardinality of sum-dominant subsets of an arithmetic progression. 
    Fix $n\ge 16$. Let $N$ be the cardinality of the largest sum-dominant subset(s) of $\{0,1,\ldots,n-1\}$ that contain(s) $0$ and $n-1$. Then $n-7\le N\le n-4$; that is, from an arithmetic progression of length $n\ge 16$, we need to discard at least $4$ and at most $7$ elements (in a clever way) to have the largest sum-dominant set(s). 
    \item Let $R\in \mathbb{N}$ have the property that for all $r\ge R$, $\{1,2,\ldots,r\}$ can be partitioned into $3$ sum-dominant subsets, while $\{1,2,\ldots,R-1\}$ cannot. Then $24\le R\le 145$. This result answers a question by the author et al. in another paper on whether we can find a stricter upper bound for $R$. 
\end{enumerate}
\noindent 2010 {\it Mathematics Subject Classification}:
Primary 11P99.

\noindent \emph{Keywords: }sum-dominant set, MSTD set, arithmetic progression.
\pagestyle{myheadings} 
\thispagestyle{empty} 
\baselineskip=12.875pt 
\vskip 30pt
\section{Introduction}
\subsection{Background and main results}
Given a finite set $A\subseteq \mathbb{N}$, define $A+A = \{a_i + a_j\,|\, a_i, a_j\in A\}$ and $A-A = \{a_i - a_j \,|\, a_i, a_j\in A\}$. The set $A$ is said to be 
\begin{itemize}
    \item \textit{sum-dominant}, if $|A+A|>|A-A|$;
    \item \textit{balanced}, if $|A+A| = |A-A|$; and
    \item \textit{difference-dominant}, if $|A+A|<|A-A|$.
\end{itemize}
Because addition is commutative, while subtraction is not, sum-dominant sets are very rare. However, it was first proved by Martin and O'Bryant \cite{MO} that as $n\rightarrow 
\infty$, the proportion of sum-dominant subsets of $\{0,1,2,\ldots,n-1\}$ is bounded below by a positive constant (about $2\cdot 10^{-7}$), which was later improved by Zhao \cite{Zh2} to about $4\cdot 10^{-4}$. However, these works used the probabilistic method and did not give explicit constructions of sum-dominant sets. Later, Miller et al. \cite{MOS} constructed a family of density $\Theta(1/n^4)$\footnote{A more refined analysis improves the bound to $\Theta(1/n^2)$ \cite{ILMZ2}.} and Zhao \cite{Zh1} gave a family of density $\Theta(1/n)$. The last few years have seen an explosion of papers exploring properties of sum-dominant sets: see \cite{FP, ILMZ1, Ma, Na2, Ru1, Ru2, Ru3} for
history and overview, \cite{He,MS,MOS,Na2,Zh1} for explicit constructions, \cite{CLMS2, HM, MO, Zh2} for positive lower bounds for the percentage of
sum-dominant sets, \cite{ILMZ2,MPR} for generalized sum-dominant sets,
and \cite{AMMS,CLMS1,CNMXZ,MV,Zh2} for extensions to other settings. 

We know that numbers from an arithmetic progression do not form a sum-dominant set. (We prove in the next section.) It is natural to ask whether numbers from the union of several arithmetic progressions produce a sum-dominant set. Our first result is that the union of two arithmetic progressions with the same common difference is not sum-dominant. 
\begin{thm}\label{maintheo}
The union of two arithmetic progressions $P_1$ and $P_2$ (with the same common difference) is not sum-dominant. 
\end{thm}
This result partially proves the conjecture by the author of the current paper \cite{Chu} that the union of any two arbitrary arithmetic progressions is not sum-dominant. (The author is motivated by the anonymous referee's comment that the conjecture was marvelous and tantalizing.) Note that $\{0,2\}\cup \{3,7,11,\ldots,4k-1\}\cup \{4k,4k+2\}$ for $k\ge 5$ is sum-dominant \cite{Na2}, and the set is the union of three arithmetic progressions. Hence, \cite[Conjecture 17]{Chu} is the most we can do. 

Our next result concerns the cardinality of a sum-dominant set. Hegarty \cite{He} proved that a sum-dominant set must have at least $8$ elements with the help of computers. The author of the current paper provided a human-understandable proof that a sum-dominant set must have at least $7$ elements \cite{Chu, Chu2}. Another natural question is about the largest cardinality of a sum-dominant set. It is well-known that a sum-dominant set can be arbitrarily large, so we put a restriction on the size of the set to have the following result 
\begin{thm}\label{notallset}
Fix $n\ge 16$. Let $N$ be the cardinality of the largest sum-dominant subset(s) of $\{0,1,\ldots,n-1\}$ that contain(s) $0$ and $n-1$. Then $n-7\le N\le n-4$. 
\end{thm}
The theorem implies that from an arithmetic progression of length at least $16$, we need to discard at least $4$ elements and not more than $7$ elements (in a clever way) to have the largest sum-dominant set(s). A corollary is that if we want to search for all sum-dominant subsets of $\{0,1,\ldots,n-1\}$, we only need to look for subsets of size between $8$ and $n-4$.
\begin{conj}\label{discard7}
Fix $n\ge 16$. Let $N$ be the cardinality of the largest sum-dominant subset(s) of $\{0,1,\ldots,n-1\}$ that contain(s) $0$ and $n-1$. Then $N = n-7$. 
\end{conj}
We run a computer program to find that the conjecture holds for all $16\le n\le 34$. For $n\le 14$, $N$ does not exist. For $n = 15$, $N=9$, corresponding to the set $\{0, 1, 2, 4, 5, 9, 12, 13, 14\}$; that is, we discard $6$ elements. 

Our final result is related to the partition of an arithmetic progression into sum-dominant subsets. Asada et al. proved that as $r\rightarrow \infty$, the proportion of $2$-decompositions of $\{1,2,\ldots,r\}$ into sum-dominant subsets is bounded below by a positive constant \cite{AMMS}. Continuing the work, the author of the current paper with Luntzlara, Miller, and Shao proved that it is possible to partition $\{1,2\ldots,r\}$ (for $n$ sufficiently large) into $k\ge 3$ sum-dominant subsets. By defining $R$ to be the smallest integer such that for all $r\ge R$, $\{1,2,\ldots,r\}$ can be $k$-decomposed into MSTD subsets, while $\{1,2,\ldots,R-1\}$ cannot, the authors established rough lower and upper bounds for $R$. However, the upper bound when $k=3$ is very loose because it depends on a sum-dominant subset $A$ with $|A+A|-|A-A|\ge 10|A|$. 
\begin{thm}\label{3-partition}
Let $R\in \mathbb{N}$ have the property that for $r\ge R$, $\{1,2,\ldots,r\}$ can be partitioned into $3$ sum-dominant subsets, while $\{1,2,\ldots,R-1\}$ cannot. Then $24\le R\le 145$. 
\end{thm}
This theorem answers a question raised by the author of the current paper et al. about whether we can find a more efficient way to decompose $\{1,2,\ldots,r\}$ into $3$ sum-dominant sets. We find a smaller upper bound by a new way of partitioning $\{1,2,\ldots,n\}$ into $3$ sum-dominant subsets. Our construction is similar to that of Miller et al. \cite{MOS} and utilizes the fact that their construction allows a long run of missing elements. The long run of missing elements is where we can insert a fixed sum-dominant set in. 

%%%%%%%%%%%%%%%%%%%%%%%%%%%%%%%%%%%%%%%%%%%%%%%%%%%%%%%%%%%%%%%%%%%%%%%%%%%%%%%%%%%%%%%%%%%%%%%%%%%%%%%%%%%%%%%%%%%%%%%%%%%%%%%%%%%%%%%%%%%%%%%%%%%%
%%%%%%%%%%%%%%%%%%%%%%%%%%%%%%%%%%%%%%%%%%%%%%%%%%%%%%%%%%%%%%%%%%%%%%%%%%%%%%%%%%%%%%%%%%%%%%%%%%%%%%%%%%%%%%%%%%%%%%%%%%%%%%%%%%%%%%%%%%%%%%%%%%%%
%%%%%%%%%%%%%%%%%%%%%%%%%%%%%%%%%%%%%%%%%%%%%%%%%%%%%%%%%%%%%%%%%%%%%%%%%%%%%%%%%%%%%%%%%%%%%%%%%%%%%%%%%%%%%%%%%%%%%%%%%%%%%%%%%%%%%%%%%%%%%%%%%%%%
\subsection{Notation}
We introduce some notation. Let $A$ and $B$ be sets. We write $A\rightarrow B$ to mean the introduction of elements in $A$ to $B$. We also use a different notation to write a set, which was first introduced by Spohn \cite{Sp}. Given a set $S = \{m_1, m_2, \ldots, m_n\}$, we arrange its elements in increasing order and find the differences between two consecutive numbers to form a sequence. Suppose that $m_1 < m_2 < \cdots < m_n$, then our sequence is $m_2 - m_1, m_3 - m_2, m_4 -
m_3, \ldots , m_n - m_{n-1}$, and we represent $S = (m_1\,|\,m_2 - m_1, m_3 - m_2, m_4 - m_3, \ldots , m_n - m_{n-1}) = (m_1\,|\,a_1,\ldots,a_{n-1})$, where $a_i = m_{i+1}-m_i$. Any difference in $S-S$ must be equal to at least a sum $a_i+\cdots+a_j$ for some $1\le i\le j\le n-1$. Take $S = \{3, 2, 15, 10, 9\}$, for example. We arrange the elements in increasing order to have $2$, $3$, $9$, $10$, $15$, form a sequence by looking at the difference between two consecutive numbers: $1$, $6$, $1$, $5$, and write $S = (2\,|\,1, 6, 1, 5)$. All information about a set is preserved in this notation. 

An arithmetic progression is a sequence of the form $(a,a+d,a+2d,a+3d,\ldots,a+kd)$ for any arbitrary numbers $a$, $k$, and the common difference $d$. Because sum-dominance is preserved under affine transformations, we can safely assume that our arithmetic progressions contain nonnegative numbers with $1$ being the common difference. To emphasize, all arithmetic progressions we consider will have nonnegative numbers and have the same common difference, which is $1$. 
%%%%%%%%%%%%%%%%%%%%%%%%%%%%%%%%%%%%%%%%%%%%%%%%%%%%%%%%%%%%%%%%%%%%%%%%%%%%%%%%%%%%%%%%%%%%%%%%%%%%%%%%%%%%%%%%%%%%%%%%%%%%%%%%%%%%%%%%%%%%%%%%%%%%
%%%%%%%%%%%%%%%%%%%%%%%%%%%%%%%%%%%%%%%%%%%%%%%%%%%%%%%%%%%%%%%%%%%%%%%%%%%%%%%%%%%%%%%%%%%%%%%%%%%%%%%%%%%%%%%%%%%%%%%%%%%%%%%%%%%%%%%%%%%%%%%%%%%%
%%%%%%%%%%%%%%%%%%%%%%%%%%%%%%%%%%%%%%%%%%%%%%%%%%%%%%%%%%%%%%%%%%%%%%%%%%%%%%%%%%%%%%%%%%%%%%%%%%%%%%%%%%%%%%%%%%%%%%%%%%%%%%%%%%%%%%%%%%%%%%%%%%%%

%%%%%%%%%%%%%%%%%%%%%%%%%%%%%%%%%%%%%%%%%%%%%%%%%%%%%%%%%%%%%%%%%%%%%%%%%%%%%%%%%%%%%%%%%%%%%%%%%%%%%%%%%%%%%%%%%%%%%%%%%%%%%%%%%%%%%%%%%%%%%%%%%%%
%%%%%%%%%%%%%%%%%%%%%%%%%%%%%%%%%%%%%%%%%%%%%%%%%%%%%%%%%%%%%%%%%%%%%%%%%%%%%%%%%%%%%%%%%%%%%%%%%%%%%%%%%%%%%%%%%%%%%%%%%%%%%%%%%%%%%%%%%%%%%%%%%%%
%%%%%%%%%%%%%%%%%%%%%%%%%%%%%%%%%%%%%%%%%%%%%%%%%%%%%%%%%%%%%%%%%%%%%%%%%%%%%%%%%%%%%%%%%%%%%%%%%%%%%%%%%%%%%%%%%%%%%%%%%%%%%%%%%%%%%%%%%%%%%%%%%%%
\section{Important Results}
We use the definition of a symmetric set given by Nathanson \cite{Na1}: a set $A$ is symmetric if there exists a number $a$ such that $a-A = A$. If so, we say that the set $A$ is symmetric about $a$. The following proposition was proved by Nathanson \cite{Na1}. 
\begin{prop}\label{sym}
A symmetric set is balanced. 
\end{prop}
\begin{proof}
Let $A$ be a symmetric set about $a$. We have $|A+A| = |A+(a-A)| = |a+(A-A)| = |A-A|$. Hence, $A$ is balanced. 
\end{proof}
Though symmetric sets are not sum-dominant, adding a few numbers into these sets (in a clever way) can produce sum-dominant sets. Examples of such a technique were provided by Hegarty \cite{He} and Nathanson \cite{Na2}. 
\begin{cor}\label{arithnotsumdominant}
A set of numbers from an arithmetic progression is not sum-dominant. 
\end{cor}
Note that a set of numbers from an arithmetic progression is symmetric about the sum of the maximum and the minimum of the arithmetic progression.
For example, the set $E= \{3,5,7,9,11\}$ is symmetric about $14$. The following lemma is proved by Macdonald and Street \cite{MS}.
\begin{lem}\label{mcdonald}Given a finite set $A = (0\,|\,a_1,a_2,\ldots,a_n)$, the following claims hold.
\begin{itemize}
    \item [(1)] If $a_i\le 2$ for all $i$, then $A$ is not sum-dominant. 
    \item [(2)] If $a_i\in \{1,m\}$ and the first and last times that $1$ occurs as a difference, it occurs in a block of at least $m-1$ consecutive differences, then $A$ is not sum-dominant. 
\end{itemize}
\end{lem}
The following lemma is trivial but very useful in our proof of Theorem \ref{maintheo}. 

\begin{lem}\label{atmost3}
The following claims hold.
\begin{itemize}
    \item [(1)] Given an arithmetic progression $P_1$, $\{\max P_1+1\}\rightarrow P_1$ gives $2$ new sums. 
    \item [(2)] Given arithmetic progressions $P_1$ and $P_2$, $\{\max P_1+1\}\rightarrow (P_1\cup P_2)$ gives at most $3$ new sums. 
\end{itemize}
\end{lem}
\begin{proof}
We first prove item 1. Without loss of generality, assume $P_1=\{0,1,\ldots,n\}$ for some $n\ge 0$. Denote $Q_1 = P_1\cup \{n+1\}$. Then $P_1+P_1 = \{0,1,\ldots,2n\}$ and $Q_1 + Q_1 = \{0,1,\ldots,2n+2\}$. Clearly, $|Q_1+Q_1|-|P_1+P_1| = 2$. 

We proceed to prove item 2. New sums come from the interactions of $\max P_1+1$ with $P_1$, with $P_2$, and with itself. By item 1, the interactions of $\max P_1+1$ with $P_1$ and itself give at most $2$ new sums. We consider the interactions of $\max P_1+1$ with $P_2$. We have
$$(\max P_1+1+P_2)\backslash (\max P_1+P_2) = \{\max P_1+\max P_2+1\}.$$
Therefore, the interactions of $\{\max P_1+1\}$ with $P_2$ gives at most $1$ new sum. In total, we have at most $3$ new sums, as desired. 
\end{proof}
%%%%%%%%%%%%%%%%%%%%%%%%%%%%%%%%%%%%%%%%%%%%%%%%%%%%%%%%%%%%%%%%%%%%%%%%%%%%%%%%%%%%%%%%%%%%%%%%%%%%%%%%%%%%%%%%%%%%%%%%%%%%%%%%%%%%%%%%%%%%%%%%%%%
%%%%%%%%%%%%%%%%%%%%%%%%%%%%%%%%%%%%%%%%%%%%%%%%%%%%%%%%%%%%%%%%%%%%%%%%%%%%%%%%%%%%%%%%%%%%%%%%%%%%%%%%%%%%%%%%%%%%%%%%%%%%%%%%%%%%%%%%%%%%%%%%%%%
%%%%%%%%%%%%%%%%%%%%%%%%%%%%%%%%%%%%%%%%%%%%%%%%%%%%%%%%%%%%%%%%%%%%%%%%%%%%%%%%%%%%%%%%%%%%%%%%%%%%%%%%%%%%%%%%%%%%%%%%%%%%%%%%%%%%%%%%%%%%%%%%%%%
\section{Proof of Theorem \ref{maintheo}}
Because sum-dominance is preserved under affine transformations, without loss of generality, assume that $0=\min P_1\le \min P_2$ and $|P_1|\ge |P_2|$. Let $m_i$ and $M_i$ denote $\min P_i$ and $\max P_i$, respectively. Finally, we only consider $P_1\cap P_2 = \emptyset$ because if $P_1\cap P_2 \neq \emptyset$, $P_1\cup P_2$ is an arithmetic progression\footnote{Recall that $P_1$ and $P_2$ have the same common difference.}, which does not form a sum-dominant set by Corollary \ref{arithnotsumdominant}. Our proof considers $P_1$ as the original set and sees how $P_2
\rightarrow P_1$ changes the number of sums and differences. 

\subsection{Part I. $\mathbf{\max P_1 < \min P_2}$}

Let $k= \min P_2-\max P_1$. If $k=1$, $P_1\cup P_2$ is an arithmetic progression, not a sum-dominant set. We consider two cases corresponding to $k<1$ and $k>1$. 

%%%%%%%%%%%%%%%%%%%%%%%%%%%%%%%%%%%%%%%%%%%%%%%%%%%%%%%%%%%%%%%%%%%%%%%%%%%%%%%%%%%%%%%%%%%%%%%%%%%%%%%%%%%%%%%%%%%%%%%%%%%%%%%%%%%%%%%%%%%%%%
%%%%%%%%%%%%%%%%%%%%%%%%%%%%%%%%%%%%%%%%%%%%%%%%%%%%%%%%%%%%%%%%%%%%%%%%%%%%%%%%%%%%%%%%%%%%%%%%%%%%%%%%%%%%%%%%%%%%%%%%%%%%%%%%%%%%%%%%%%%%%%
%%%%%%%%%%%%%%%%%%%%%%%%%%%%%%%%%%%%%%%%%%%%%%%%%%%%%%%%%%%%%%%%%%%%%%%%%%%%%%%%%%%%%%%%%%%%%%%%%%%%%%%%%%%%%%%%%%%%%%%%%%%%%%%%%%%%%%%%%%%%%%
\noindent \textbf{Case I.1:} $k < 1$. We consider $P_2\rightarrow P_1$. The set of new positive and distinct differences includes
$$k \ <\  k+1 \ <\  \cdots \ <\  k+|P_1|+|P_2|-2.$$
Hence, the number of new differences is at least $2(|P_1|+|P_2|-1)$. Now, we count the number of new sums. Consider $m_2 \rightarrow P_1$. We have at most $|P_1|+1$ new sums. Due to Lemma \ref{atmost3}, $m_2+j\rightarrow P_1\cup\{m_2,\ldots, m_2+j-1\}$ gives at most $3$ new sums for all $j\ge 1$. Therefore, $P_2\rightarrow P_1$ gives at most $$(|P_1|+1)+3(|P_2|-1) \ =\ |P_1|+3|P_2|-2$$ new sums. 

Because $|P_1|\ge |P_2|$, we have
$$2(|P_1|+|P_2|-1)\ \ge \ |P_1|+3|P_2|-2,$$ and so, we do not have a sum-dominant set. 

\noindent \textbf{Case I.2:} $k>1$. If $k$ is not a multiple of $1$, then with the same reasoning as Case I.1, we are done. If $k$ is a multiple of $1$, we consider two following subcases. 

\textit{Subcase I.2.1:} $k> \max P_1$. Then $m_2\rightarrow P_1$ gives $|P_1|$ new positive differences
    $$m_2-\max P_1\ <\ m_2 - \max P_1+1\ <\ \cdots \ < \ m_2$$
    while at most $|P_1|+1$ new sums. Due to Lemma \ref{atmost3}, $m_2+j\rightarrow P_1\cup\{m_2,\ldots, m_2+j-1\}$ gives at most $3$ new sums and at least $2$ new differences $\pm (m_2+j)$ for all $j\ge 1$. Therefore, $P_2\rightarrow P_1$ gives at most $|P_1|+1+3(|P_2|-1)$ new sums while at least $2|P_1|+2(|P_2|-1)$ new differences. Because $|P_1|\ge |P_2|$, the number of new differences is not smaller than the number of new sums, and so, $P_1\cup P_2$ is not sum-dominant. 
    
\textit{Subcase I.2.2:} $k\le \max P_1$. If $|P_2|\ge k$, we are done due to item 2 Lemma \ref{mcdonald}. So, we consider $|P_2| \le k-1$. Consider $m_2\rightarrow P_1$. By \cite[Proposition 7]{Chu}, $m_2\rightarrow P_1$ gives $2k$ new differences and $k+1$ new sums. Due to Lemma \ref{atmost3}, $m_2+j\rightarrow P_1\cup\{m_2,\ldots, m_2+j-1\}$ gives at most $3$ new sums and $2$ new differences $\pm (m_2+j)$ for all $j\ge 1$. The total number of new sums is at most $k+1+3(|P_2|-1)$, while the number of new differences is at least $2k+2(|P_2|-1)$. We have
$$2k+2(|P_2|-1)-(k+1+3(|P_2|-1)) \ =\ k-|P_2| \ \ge\ 1.$$
Hence, $P_1\cup P_2$ is not sum-dominant. 

\subsection{\textbf{Part II. $\mathbf{\max P_1 > \min P_2}$}}

If $m_2-1/2\in \mathbb{Z}$, we consider $2(P_1\cup P_2)$. Because the difference between any two consecutive numbers in increasing order is either $1$ or $2$, by item 1 Lemma \ref{mcdonald}, we do not have a sum-dominant set. Hence, we assume that $m_2-1/2\notin \mathbb{Z}$. Suppose that $n<m_2<n+1$ for some $n\in P_1$. The following are new and pairwise distinct positive differences from $m_2\rightarrow P_1$
\begin{align*}
    &m_2-n\ <\ m_2-(n-1)\ <\ \cdots\ <\ m_2-0,\\
    &n+1-m_2\ <\ n+2-m_2\ <\ \cdots\ <\ \max P_1-m_2.
\end{align*}
Hence, we have at least $2|P_1|$ new differences. On the other hand, $m_2\rightarrow P_1$ gives at most $|P_1|+1$ new sums. Due to Lemma \ref{atmost3}, $m_2+j\rightarrow P_1\cup\{m_2,\ldots, m_2+j-1\}$ gives at most $3$ new sums and at least $2$ new differences $\pm (m_2+j)$ for all $j\ge 1$.  
Hence, the total number of new sums as a result of $P_2\rightarrow P_1$ is at most $$|P_1|+1 +3(|P_2|-1) \ =\  |P_1|+3|P_2|-2,$$
while the number of new differences is at least $$2|P_1|+2(|P_2|-1) \ =\ 2|P_1|+2|P_2|-2.$$
Because $|P_1|\ge |P_2|$, we have $$|P_1|+3|P_2|-2\ \le \ 2|P_1|+2|P_2|-2.$$ Therefore, $P_1\cup P_2$ is not sum-dominant. 
 
We finish our proof. 
%%%%%%%%%%%%%%%%%%%%%%%%%%%%%%%%%%%%%%%%%%%%%%%%%%%%%%%%%%%%%%%%%%%%%%%%%%%%%%%%%%%%%%%%%%%%%%%%%%%%%%%%%%%%%%%%%%%%%%%%%%%%%%%%%%%%%%%%%%%%%%%%%%
%%%%%%%%%%%%%%%%%%%%%%%%%%%%%%%%%%%%%%%%%%%%%%%%%%%%%%%%%%%%%%%%%%%%%%%%%%%%%%%%%%%%%%%%%%%%%%%%%%%%%%%%%%%%%%%%%%%%%%%%%%%%%%%%%%%%%%%%%%%%%%%%%%
%%%%%%%%%%%%%%%%%%%%%%%%%%%%%%%%%%%%%%%%%%%%%%%%%%%%%%%%%%%%%%%%%%%%%%%%%%%%%%%%%%%%%%%%%%%%%%%%%%%%%%%%%%%%%%%%%%%%%%%%%%%%%%%%%%%%%%%%%%%%%%%%%%
%%%%%%%%%%%%%%%%%%%%%%%%%%%%%%%%%%%%%%%%%%%%%%%%%%%%%%%%%%%%%%%%%%%%%%%%%%%%%%%%%%%%%%%%%%%%%%%%%%%%%%%%%%%%%%%%%%%%%%%%%%%%%%%%%%%%%%%%%%%%%%%%%%
\section{Proof of Theorem \ref{notallset}}
\begin{lem}\label{lowerboundof7}
For $m\ge 9$, the set
\begin{align*}K &\ =\ \{0,1,\ldots,m+7\}\backslash \{3,5,6,m+1,m+2,m+3,m+5\}\\
&\ =\ \{0,1,2,4\}\cup\{7,8,\ldots,m\}\cup\{m+4,m+6,m+7\}\end{align*}
is sum-dominant. Note that $K$ is a sum-dominant subset of $\{0,1,\ldots,m+7\}$ after we discard $7$ numbers from the arithmetic progression. 
\end{lem}
\begin{proof}
Observe that $K-K = \{\pm 0,\pm 1,\ldots,\pm (m+7)\}\backslash \{\pm (m+1)\}$, while $K+K = \{0,1,\ldots,2m+14\}\backslash \{2m+9\}$. Hence, $|K+K|-|K-K| = 1$. 
\end{proof}

We now prove Theorem \ref{notallset}.
Fix $n\ge 16$. Let $N$ be the cardinality of the largest sum-dominant subset(s) of $\{0,1,\ldots,n-1\}$. Lemma \ref{lowerboundof7} proves the lower bound for $N$ in Theorem \ref{notallset}; that is, $N\ge n-7$. We proceed to show that $N\le n-4$. 

If $N=n$, we have the arithmetic progression $\{0,1,\ldots,n-1\}$, which is not sum-dominant. 

If $N=n-1$, we do not have a sum-dominant set due to item 1 Lemma \ref{mcdonald}.

If $N=n-2$, we have two cases. If the two missing numbers are not next to each other, we do not have a sum-dominant set due to item 1 Lemma \ref{mcdonald}. If the two missing numbers are next to each other, we do not have a sum-dominant set due to Theorem \ref{maintheo}.

If $N=n-3$, we have three cases. 
\begin{enumerate}
\item Case 4.1: If the three missing numbers are consecutive, then we do not have a sum-dominant set due to Theorem \ref{maintheo}.

\item Case 4.2: If no two numbers are next to each other, then we do not have a sum-dominant set due to item 1 Lemma \ref{mcdonald}.

\item Case 4.3: Two numbers are next to each other, while the other is not next to any of these numbers. Let the two numbers that are next to each other be $k$ and $k+1$ for some $k\ge 0$. Without loss of generality, assume that the third number is $k+p$ such that $k+p>k+2$.
\begin{enumerate}
\item If $k=0$, we have the set $\{2,3,\ldots,n-1\}\backslash \{k+p\}$, which is not sum-dominant due to item 1 Lemma \ref{mcdonald}.
\item If $k+p = n-1$, we are back to the case $N=n-2$.
\item Suppose that $k>0$ and $k+p<n-1$. We have all differences in $\{0,1,\ldots,n-1\}\backslash\{k,k+1,k+p\}$ by looking at $\{0,1,\ldots,n-1\}\backslash\{k,k+1,k+p\}-0$. If we do not have any missing differences, then we are done. 

If $k=1$, because $n\ge 16$ and we miss only $3$ numbers, it must be that we have three consecutive numbers in our set. So, we have differences of $1$ and $2$, and so, $k$ and $k+1$ are in the difference set. Hence, we miss at most $2$ differences, which are $\pm (k+p)$. However, we also miss at least $2$ sums, which are $1$ and $2$. Therefore, we do not have a sum-dominant set. 

If $k=2$, then $1$ is in our set. We have $k+p$ by looking at $(k+p+1)-1$ and $k+1$ by looking at $(k+2)-1$. Because $n\ge 16$ and we miss only $3$ numbers, it must be that we have three consecutive numbers in our set. So, we have a difference of $2$, and so, $k$ is in the difference set. We are done. 

If $k>2$, then $1$ and $2$ are in our set. We have $k+p$ by looking at $(k+p+1)-1$, $k+1$ by looking at $(k+2)-1$, and $k$ by looking at $(k+2)-2$. We are done. 

\end{enumerate}
\end{enumerate}

%%%%%%%%%%%%%%%%%%%%%%%%%%%%%%%%%%%%%%%%%%%%%%%%%%%%%%%%%%%%%%%%%%%%%%%%%%%%%%%%%%%%%%%%%%%%%%%%%%%%%%%%%%%%%%%%%%%%%%%%%%%%%%%%%%%%%%%%%%%%%%%%%%
%%%%%%%%%%%%%%%%%%%%%%%%%%%%%%%%%%%%%%%%%%%%%%%%%%%%%%%%%%%%%%%%%%%%%%%%%%%%%%%%%%%%%%%%%%%%%%%%%%%%%%%%%%%%%%%%%%%%%%%%%%%%%%%%%%%%%%%%%%%%%%%%%%
%%%%%%%%%%%%%%%%%%%%%%%%%%%%%%%%%%%%%%%%%%%%%%%%%%%%%%%%%%%%%%%%%%%%%%%%%%%%%%%%%%%%%%%%%%%%%%%%%%%%%%%%%%%%%%%%%%%%%%%%%%%%%%%%%%%%%%%%%%%%%%%%%%
%%%%%%%%%%%%%%%%%%%%%%%%%%%%%%%%%%%%%%%%%%%%%%%%%%%%%%%%%%%%%%%%%%%%%%%%%%%%%%%%%%%%%%%%%%%%%%%%%%%%%%%%%%%%%%%%%%%%%%%%%%%%%%%%%%%%%%%%%%%%%%%%%%
\section{Proof of Theorem \ref{3-partition}}
We will use the construction discussed in \cite[Theorem 1.1]{CLMS1} to partition $\{1,2,\ldots,r\}$ into $3$ sum-dominant subsets. Following the construction, we fix $n = k = 20$ and set
\begin{align*}
    L_1 &\ =\ \{1,2,3,4,8,9,11,13,14,15,20\},\\
    R_1 &\ =\ \{21,26,27,28,31,33,37,38,39,40\},\\
    L_2 &\ =\ \{5,6,7,10,12,16,17,18,19\},\\
    R_2 &\ =\ \{22,23,24,25,29,30,32,34,35,36\}.
\end{align*}
Note that in \cite[Theorem 1.1]{CLMS1}, $A_1 = L_1\cup R_1$ and $A_2 = L_2\cup R_2$. 
Pick $m\ge 21$. Set
\begin{align*}
    R'_1 &\ =\ R_1+m+84,\\
    R'_2 &\ =\ R_2 + m+84,\\
    O_{11} &\ =\ \{24\}\cup \{25,27,29,\ldots,61\}\cup\{62\},\\
    O_{12} &\ =\ \{63+m\}\cup \{64+m, 66+m,68+m,\ldots,100+m\}\cup\{101+m\},\\
    O_{21} &\ =\ \{21,22,23\}\cup \{26,28,30,\ldots,60\}\cup \{63,64,65\},\\
    O_{22} &\ =\ \{60+m, 61+m, 62+m\}\cup \{65+m,67+m,\ldots,99+m\}\\
    &\ \cup\ \{102+m,103+m,104+m\}.
\end{align*}
Let $M_1\subseteq \{66,67,\ldots,59+m\}\backslash \{66,68,69,70,73,77,78,80\}$ such that within $M_1$, there exists a sequence of pairs of consecutive elements, where consecutive pairs are not more than $39$ apart and the sequence starts with a pair in $\{66,67,\ldots,101\}$ and ends with a pair in $\{24+m,25+m,\ldots,59+m\}$. Let $M_2\subseteq \{66,67,\ldots,59+m\}$ such that within $M_2$, there exists a sequence of triplets of consecutive elements, where consecutive triplets are not more than $40$ apart and the sequence starts with a triplet in $\{66,67,\ldots,105\}$ and ends with a triplet in $\{20+m,21+m,\ldots,59+m\}$. Also, $M_1\cap M_2 = \emptyset$ and $M_1\cup M_2 = \{66,67,\ldots,59+m\}\backslash \{66,68,69,70,73,77,78,80\}$. Then
\begin{align*}
    A'_1 &\ =\ L_1\cup O_{11} \cup M_1 \cup O_{12} \cup R'_1\\
    A'_2 &\ =\ L_2 \cup O_{21}\cup M_2\cup O_{22} \cup R'_2
\end{align*}
are both sum-dominant and along with $S =  \{66,68,69,70,73,77,78,80\}$ partition $\{1,124+m\}$. 
\begin{exa}\label{upperboundR}
Let $m = 21$. Set 
\begin{align*}
M_1 &\ =\ \{71,72\} \\
M_2 &\ =\ \{67,74,75,76,79\}.
\end{align*}
We partition $\{1,145\}$ into three following sum-dominant sets 
\begin{align*}
 A'_1 &\ =\ L_1\cup O_{11} \cup M_1 \cup O_{12} \cup R'_1\\
 &\ =\ \{1,2,3,4,8,9,11,13,14,15,20,24\}\cup \{25,27,29,\ldots,61\}\\&\ \cup\ \{62,71,72,84\}\cup\{85,87,\ldots,121\}\\ &\ \cup\ \{122,126,131,132,133,136,138,142,143,144,145\}\\
 &\ \mbox{ with } |A'_1+A'_1|-|A'_1-A'_1| = 2,\\
 A'_2 &\ =\ L_2 \cup O_{21}\cup M_2\cup O_{22} \cup R'_2\\
 &\ =\ \{5,6,7,10,12,16,17,18,19, 21,22,23\}\cup \{26,28,30,\ldots,60\}\\
 &\ \cup\ \{63,64,65,67,74,75,76,79,81,82,83\}\cup\{86,88,\ldots,120\}\cup\{123,124,125\}\\
 &\ \cup\ \{127,128,129,130,134,135,137,139,140,141\}\\
 &\ \mbox{ with } |A'_2+A'_2|-|A'_2-A'_2| = 2,\\
 S &\ =\ \{66,68,69,70,73,77,78,80\} \mbox{ with } |S+S|-|S-S| = 1.
\end{align*}
\end{exa}
Example \ref{upperboundR} proves the upper bound of $145$ for $R$ in our Theorem \ref{3-partition}.
%%%%%%%%%%%%%%%%%%%%%%%%%%%%%%%%%%%%%%%%%%%%%%%%%%%%%%%%%%%%%%%%%%%%%%%%%%%%%%%%%%%%%%%%%%%%%%%%%%%%%%%%%%%%%%%%%%%%%%%%%%%%%%%%%%%%%%%%%%%%%%%%%%
%%%%%%%%%%%%%%%%%%%%%%%%%%%%%%%%%%%%%%%%%%%%%%%%%%%%%%%%%%%%%%%%%%%%%%%%%%%%%%%%%%%%%%%%%%%%%%%%%%%%%%%%%%%%%%%%%%%%%%%%%%%%%%%%%%%%%%%%%%%%%%%%%%
%%%%%%%%%%%%%%%%%%%%%%%%%%%%%%%%%%%%%%%%%%%%%%%%%%%%%%%%%%%%%%%%%%%%%%%%%%%%%%%%%%%%%%%%%%%%%%%%%%%%%%%%%%%%%%%%%%%%%%%%%%%%%%%%%%%%%%%%%%%%%%%%%%
%%%%%%%%%%%%%%%%%%%%%%%%%%%%%%%%%%%%%%%%%%%%%%%%%%%%%%%%%%%%%%%%%%%%%%%%%%%%%%%%%%%%%%%%%%%%%%%%%%%%%%%%%%%%%%%%%%%%%%%%%%%%%%%%%%%%%%%%%%%%%%%%%%
%%%%%%%%%%%%%%%%%%%%%%%%%%%%%%%%%%%%%%%%%%%%%%%%%%%%%%%%%%%%%%%%%%%%%%%%%%%%%%%%%%%%%%%%%%%%%%%%%%%%%%%%%%%%%%%%%%%%%%%%%%%%%%%%%%%%%%%%%%%%%%%%%%
%%%%%%%%%%%%%%%%%%%%%%%%%%%%%%%%%%%%%%%%%%%%%%%%%%%%%%%%%%%%%%%%%%%%%%%%%%%%%%%%%%%%%%%%%%%%%%%%%%%%%%%%%%%%%%%%%%%%%%%%%%%%%%%%%%%%%%%%%%%%%%%%%%
%%%%%%%%%%%%%%%%%%%%%%%%%%%%%%%%%%%%%%%%%%%%%%%%%%%%%%%%%%%%%%%%%%%%%%%%%%%%%%%%%%%%%%%%%%%%%%%%%%%%%%%%%%%%%%%%%%%%%%%%%%%%%%%%%%%%%%%%%%%%%%%%%%
%%%%%%%%%%%%%%%%%%%%%%%%%%%%%%%%%%%%%%%%%%%%%%%%%%%%%%%%%%%%%%%%%%%%%%%%%%%%%%%%%%%%%%%%%%%%%%%%%%%%%%%%%%%%%%%%%%%%%%%%%%%%%%%%%%%%%%%%%%%%%%%%%%
\section{Future Research}
We end with a list of questions for future research.
\begin{itemize}
    \item Is Conjecture \ref{discard7} correct?
    \item Is Conjecture \cite[Conjecture 17]{Chu} correct? 
    \item Is it true that for every fixed $k\ge 2$, as $r\rightarrow \infty$, the proportion of $k$-decompositions of $\{1,2,\ldots,r\}$ into sum-dominant subsets is bounded below by a positive constant? 
\end{itemize}
%%%%%%%%%%%%%%%%%%%%%%%%%%%%%%%%%%%%%%%%%%%%%%%%%%%%%%%%%%%%%%%%%%%%%%%%%%%%%%%%%%%%%%%%%%%%%%%%%%%%%%%%%%%%%%%%%%%%%%%%%%%%%%%%%%%%%%%%%%%%%%%%%%
%%%%%%%%%%%%%%%%%%%%%%%%%%%%%%%%%%%%%%%%%%%%%%%%%%%%%%%%%%%%%%%%%%%%%%%%%%%%%%%%%%%%%%%%%%%%%%%%%%%%%%%%%%%%%%%%%%%%%%%%%%%%%%%%%%%%%%%%%%%%%%%%%%
%%%%%%%%%%%%%%%%%%%%%%%%%%%%%%%%%%%%%%%%%%%%%%%%%%%%%%%%%%%%%%%%%%%%%%%%%%%%%%%%%%%%%%%%%%%%%%%%%%%%%%%%%%%%%%%%%%%%%%%%%%%%%%%%%%%%%%%%%%%%%%%%%%
%%%%%%%%%%%%%%%%%%%%%%%%%%%%%%%%%%%%%%%%%%%%%%%%%%%%%%%%%%%%%%%%%%%%%%%%%%%%%%%%%%%%%%%%%%%%%%%%%%%%%%%%%%%%%%%%%%%%%%%%%%%%%%%%%%%%%%%%%%%%%%%%%%

\vskip 20pt
{\smallit Department of Mathematics, University of Illinois at Urbana-Champaign, Urbana, IL 61820, USA}\\
{\tt chuh19@mail.wlu.edu}\\ 

\end{document}